\documentclass[reqn,9pt]{article}
\usepackage[flushleft]{threeparttable}
\pdfoutput=1
\usepackage[square,sort,comma,numbers]{natbib}
\usepackage{tikz}
\usepackage{graphicx}
\usepackage{amsmath, amssymb, amsthm}
\usepackage{latexsym}
\usepackage{rotate}
\usepackage{lscape}
\usepackage{color}
\usepackage{graphicx}
\usepackage[english]{babel}
\usepackage[latin1]{inputenc}
\usepackage[autostyle, english = american]{csquotes}
\usepackage{listings}

\newtheorem{theorem}{\bf Theorem}[section]

\newtheorem{lemma}{\bf Lemma}[section]

\newtheorem{remark}{\bf Remark}[section]
\newtheorem{definition}{\bf Definition}[section]
\newtheorem{example}{\bf Example }[section]

\begin{document}
\title{\bf Stochastic comparisons of record values based on their relative
aging}
\author{Mohamed Kayid$^{\dag,}$\thanks{Corresponding author, Email address: drkayid@ksu.edu.sa}\\
~~\\
{\small Department of Statistics and Operations Research}\\
{\small College of Science, King Saud University, P.O. Box 2455, Riyadh 11451, Saudi Arabia}\\
{\small ~~}\\
}
\date{}
\maketitle

\begin{abstract}
In this paper we examine some relative orderings of upper and lower records. It is shown that if $m> n$, the $m$th upper record ages faster than the $n$th upper record, where the data sets come from a sequence of independent and identically distributed observations from a  continuous distribution. Sufficient conditions are also obtained to see whether the $m$th
 upper record arisen from a continuous distribution ages faster in terms of the relative hazard rate than the $n$th upper record arisen from another continuous distribution. It is also shown that the reversed hazard rate of the $m$th lower record
 decreases faster than the reversed hazard rate of the $n$th lower record, when $m> n$. Preservation property of the
 relative reversed hazard rate order at lower record values is investigated. Several examples are presented to examine the results.

\end{abstract}

\vskip 2mm

\section{Introduction}
\label{seq1}
Order statistics result from a random sample of finite size, independent of the order of the observations in the sample. The extreme values of the sample are referred to as minimum order statistic and maximum order statistic. Record values are the extreme values of a sequence of random samples that converge to the extremes of the distribution under consideration. Formally, successive observations of a single random variable, called a sequence of random variables, yield two statistics, namely the upper record value and the lower record value. Extreme amounts of a random sequence of random variables are available as record values and, as a result, record values are important statistics that are useful to show the extremes of the distribution of observations. The record values have been extensively studied in the literature thus far. Because they are found in so many academic fields such as climatology, athletics, medicine, transportation, industry and so on, records are very popular. These records serve as time capsules. Through the history of records, we can examine how science and technology have developed and judge humanity based on historical achievements in a number of fields. Numerous long-term records have inspired the development of various mathematical models that reflect related record-keeping processes and predict future record outcomes.
Records have been thoroughly studied in the literature. Considerable progress has been made in the field of stochastic orderings of record values. We cite the following sources for a selection of results: Khaledi \cite{Khaledi2005}, Hu and Zhuang \cite{Hu2006}, Raqab and Amin \cite{Raqab1996}, Ahmadi and Arghami \cite{Ahmadi2001}, Belzunce et al. \cite{Belzunce2001}, Kayid et al. \cite{Kayid2021}, Esna-Ashari et al. \cite{Esna-Ashari2023}, Aswin et al. \cite{Aswin2023} and references therein. Raqab and Asadi \cite{Raqab2008} have published more recent work on the mean residual lives of data sets considering reliability, aging aspects and other relevant properties. For more information on distribution theory and its applicability to different types of data modeled by record values, see Ahsanullah \cite{Ahsanullah1995} and Arnold et al. \cite{Arnold1998}. \bigskip

Aging is a universal phenomenon that affects both mechanical systems and living beings. The age of a living organism is the time $t$ at which this organism is still alive and functioning. The aging process of an object is usually due to the wear and tear of the object. It usually explains how a system or living being gets better or worse as it gets older. In recent decades, researchers have been intensively studying stochastic aging. Numerous ideas on stochastic aging have been developed in the literature to characterize various aging features of a system, such as the increasing failure rate (IFR), the increasing failure rate average (IFRA), etc. There are three categories of aging: no aging, negative aging and positive aging. Two publications that provide a quick overview of different aging principles are Barlow and Proschan \cite{Barlow1975} and Lai and Xie \cite{Lai2006}. Comparative aging is a useful idea that describes how one system ages relative to another and is similar to these other aging terms.\bigskip

The term ``relative aging" refers to the aging of one system compared to another. The commands for faster aging are those that compare the respective ages of the two systems. The relative age of two systems can be described by two additional sets of stochastic orders. The convex transformation order, the quantitative mean inactive time order, the star-shaped order, the superadditive order, the DMRL (decreasing mean residual life) order, the s-IFR order, and other so-called transformation orders belong to the first set of stochastic orders, which characterize whether one system ages faster than another in terms of increasing failure rate, increasing failure rate on average, new better than used, etc. Barlow and Proschan \cite{Barlow1975}, Bartoszewicz \cite{Bartoszewicz1985}, Deshpande and Kochar \cite{Deshpande1983}, Kochar and Wiens \cite{Kochar1987}, Arriaza et al \cite{Arriaza2017}, Nanda et al \cite{Nanda2017} and the references therein provide a thorough treatment of these orderings. The monotonicity of the ratios of certain reliability measures, such as the mean residual life function, the inverted hazard rate function and the hazard rate function, defines the second group of stochastic orders, often referred to as faster aging orders. We recommend reading Kalashnikov and Rachev \cite{Kalashnikov1986}, Sengupta and Deshpande \cite{Sengupta1994}, Di Crescenzo \cite{Di2000}, Finkelstein \cite{Finkelstein2006}, Razaei et al. \cite{R2015}, Hazra and Nanda \cite{Hazra2016}, Misra et al \cite{Misra2017}, Kayid et al \cite{Kayid2017} and Misra and Francis \cite{Misra2018} for information on the foundations and applicability of these orders. \\

In reliability and survival analysis, the proportional hazard (PH) model, also known as Cox's PH model (see Cox \cite{Cox1972}), is often used to analyze failure time data. Later, other models were developed, including the proportional mean remaining life model, the proportional odds model, the proportional inverse hazard rate model, and others (see Marshall and Olkin \cite{Marshall2007}, Lai and Xie \cite{Lai2006}, and Finkelstein \cite{Finkelstein2008}). The phenomenon of crossing hazards and, further, crossing mean residual lifetimes has been observed in numerous real-world scenarios (see Mantel and Stablein \cite{Mantel1988}, Pocock et al. \cite{Pocock1982} and Champlin et al. \cite{Champlin1983}). Based on the idea of relative aging, Kalashnikov and Rachev \cite{Kalashnikov1986} developed a stochastic ordering (called faster aging ordering in the hazard rate) to address this problem of crossing hazard rates. In fact, this strategy could be seen as a solid replacement for the PH model. Sengupta and Deshpande \cite{Sengupta1994} give a thorough analysis of this order. They also presented two other related types of stochastic ordering.
Subsequently, Razaei et al. \cite{R2015} offered a similar stochastic ordering in the form of the reversed hazard rate functions, while  Finkelstein \cite{Finkelstein2006} proposed a stochastic ordering based on the mean residual life functions characterizing the relative ages of two life distributions. Hazra and Nanda \cite{Hazra2016} presented some generalized orderings in this direction.\\

However, the problem of the relative aging ordering of record values has not yet been considered in the literature. The aim of the present study is to initiate such an investigation in order to develop preservation properties of the relative hazard rate order and the relative reversed hazard rate order under upper records and lower records, respectively. \medskip

The structure of the paper is as follows. In Section \ref{sec:2}, we introduce some preliminary concepts and definitions. In Section \ref{sec:3},
we investigate the preservation property of the relative hazard rate order under upper records. In Section \ref{sec:4}, preservation property of the relative reversed hazard rate order under lower records is investigated. In Section \ref{sec:5}, we conclude the paper with further remarks and illustrations on current and future research.

\section{Preliminaries}
\label{sec:2}
In this section, we bring some preliminaries that will be used throughout the paper. The definitions of the ageing faster orders utilized in our paper are provided below.

\begin{definition}
Let $X$ and $Y$ be two absolutely continuous random variables with hazard rate functions $h_{X}$ and $h_{Y}$, respectively, and reversed hazard rate functions $r_{X}$ and $r_{Y}$, respectively. It is then said that $X$ ages faster than $Y$ in

\begin{itemize}
  \item [(i)] hazard rate, denoted by $X\preceq_{c} Y$, if
  $$\frac{h_X(t)}{h_Y(t)}~~\textmd{is increasing in}~t\geq 0.$$

  \item [(ii)] reversed hazard rate, denoted by $X\preceq_{b} Y$, if
    $$\frac{r_X(t)}{r_Y(t)}~~\textmd{is decreasing in}~t\geq 0.$$
\end{itemize}
\end{definition}

Let us look at a technical system that experiences shocks like voltage peaks. A set of independent and identically distributed (i.i.d.) rvs $\{X_i,i\geq1\}$ with a common continuous cdf $F_X$, pdf $f_X$ and survival function $\bar{F}_{X}\equiv 1-F_{X}$ can then be used to simulate the shocks. The loads on the system at different points in time are represented by the shocks. The record statistics (values of the highest stresses reported so far) of this sequence are of interest to us. As we consider the lifetime of devices in our context, we thus suppose that the observations which produce record values are non-negative. Hence, the record values are consequently non-negative. The $i$-th order statistics from the first $n$ observations are labeled with the symbol $X_{i:n}$.

Next, we define the upper record values $X_{U_{n}},$ and upper record timings $\{T_n,n\geq 1\}$, respectively, as follows:
\begin{equation*}\label{eq1}
X_{U_{n}}=X_{T_n:T_n},\ n=0,1,\ldots,
\end{equation*}
where
\begin{equation*}
T_0=1,\ T_{n}=\min\{j:j>T_{n-1},X_j>X_{U_{n}}\},\ n\geq1.
\end{equation*}
It is commonly known that the pdf and sf of $X_{U_n},$ represented by $f_{X_{U_n}}(x)$ and $\bar{F}_{X_{U_n}}(x)$ respectively, are given by
\begin{equation}\label{eq2}
f_{X_{U_n}}(x)=\frac{[-\log (\bar{F}_{X}(x))]^{n-1}}{(n-1)!}f_X(x),~~ x \geq 0,
\end{equation}
and
\begin{align}\label{eq3}
\bar{F}_{X_{U_n}}(x)&=\int_{-\log (\bar{F}_X(x))}^{+\infty} \frac{u^{n-1}}{(n-1)!}e^{-u}du\notag\\
&=\bar{F}_X(x) \, \sum_{k=0}^{n-1} \frac{[-\log (\bar{F}_X(x))]^k}{k!},~~x \geq 0,
\end{align}
where the last identity is derived by using the expansion of incomplete gamma function (see e.g. \cite{Arnold2008}).
In contrast to the upper record values are the lower record values. The $n$th lower record time $L(n), n=1,2,\ldots$ with $L(1)=1$ is stated as
\begin{equation*}
L(1)=1,\ L(n+1)=\min\{j:j>L(n),X_{1:L(n)}>X_{1:j}\},\ n=1,2,\ldots.
\end{equation*}
and the $n$-th lower record is enumerated as $X_{L_n}=X_{1:L(n)}, n=1,2,\ldots.$ The pdf of $X_{L_n}$ can be acquired as
\begin{equation}\label{eq6}
f_{X_{L_n}}(x)=\frac{[-\ln (F_{X}(x))]^{n-1}}{(n-1)!}f_X(x),~~ x\geq 0.
\end{equation}
In addition, $X_{L_n}$ has cdf:
\begin{align}\label{eq7}
F_{X_{L_n}}(x)&=\int_{-\log(F_X(x))}^{+\infty} \frac{u^{n-1}}{(n-1)!}e^{-u}du \notag \\
&=F_{X}(x) \, \sum_{k=0}^{n-1} \frac{[-\ln (F_X(x))]^k}{k!};\; x\geq 0.
\end{align}

Three stochastic orders which consider magnitude of random
variables rather than their relative aging behaviours are defined below according to Shaked and Shanthikumar \cite{Shaked2007}.

\begin{definition}
Let us suppose that $X$ and $Y$ represent the lifetime of two devices. It is said that $X$ is smaller than $Y$ in the

\begin{itemize}
  \item [(i)] hazard rate order, denoted by $X\preceq_{hr} Y$, if
  $$h_X(t)\geq h_Y(t)~~\textmd{for all}~t\geq 0.$$

   \item [(ii)] reversed hazard rate order, denoted by $X\preceq_{rh} Y$, if
  $$r_X(t)\leq r_Y(t)~~\textmd{for all}~t>0.$$

  \item [(iii)] usual stochastic order, denoted by $X\preceq _{st} Y$, if
    $$\bar{F}_X(t)\leq \bar{F}_Y(t)~~\textmd{for all}~t\geq 0.$$
\end{itemize}
\end{definition}

The following definition is regarding the aging property of a life unit.\medskip

\begin{definition}
Let $X$ be a lifetime random variable with hazard rate $h_X$ and reversed hazard rate $r_X$. It is said that $X$ has

\begin{itemize}
  \item [(i)] increasing failure rate (denoted as $X\in IFR$), if $h_X(t)$ is increasing in $t\geq 0$.

   \item [(ii)] decreasing reversed hazard rate (denoted as $X\in DRHR$) if
  $r_X(t)$ is decreasing in $t>0$.
\end{itemize}
\end{definition}

The following definition is due to Karlin \cite{Karlin1968} which be used frequently in the sequel.

\begin{definition}
Let $f(x,y)$ be a non-negative function. It is said that $f$ is Totally positive of order 2 ($TP_2$) in $(x,y)\in \mathcal{X} \times \mathcal{Y}$ where
$\mathcal{X}$ and $\mathcal{Y}$ are two arbitrary subsets of $\mathbb{R}=(-\infty,+\infty)$ whenever
\begin{equation}\label{detTP2}
\begin{array}{|cc|}
f(x_1,y_1) & f(x_1,y_2) \\
f(x_2,y_1) & f(x_2,y_2) \\
\end{array}\geq 0,~~\textmd{for all}~x_1\leq x_2 \in \mathcal{X},~ \textmd{and for all}~~y_1\leq y_2 \in \mathcal{Y}.
\end{equation}
If the direction of the inequality given after determinant in (\ref{detTP2}) is reversed then it is said that $f$ is
Reverse regular of order 2 ($RR_2$) in $(x,y)\in \mathcal{X} \times \mathcal{Y}$.
\end{definition}

\section{Relative aging of upper record values}
\label{sec:3}
Suppose that $X$ and $Y$ are two non-negative rvs with absolutely continuous cdfs $F_X$ and $F_Y$, and the associated pdfs
$f_X$ and $f_Y$, respectively. We also denote by $\bar{F}_X$ and $\bar{F}_Y$ the sfs of $X$ and $Y$, respectively. Let us obtain the expression of hazard rate of $X_{U_m}$ and $Y_{U_n}$, as the upper record values of two sequences of rvs $\{X_i; i=1,2,\ldots\}$ and $\{Y_i; i=1,2,\ldots\}$, respectively. From Eq. (\ref{eq2}) and Eq. (\ref{eq3}), the hazard rate of
$X_{U_m}$ at point $t$ of time, is derived as
\begin{align}\label{hrur1}
h_{X_{U_m}}(t)&= \frac{f_{X_{U_m}}(t)}{\bar{F}_{X_{U_m}}(t)}\nonumber \\
&=\displaystyle\frac{(-\ln(\bar{F}_X(t)))^{m-1}}{(m-1)! \sum_{i=0}^{m-1}(-\ln(\bar{F}_X(t)))^{i}/i!}h_{X}(t)\nonumber \\
&=\displaystyle\frac{1}{(m-1)! \sum_{i=0}^{m-1}(-\ln(\bar{F}_X(t)))^{i-m+1}/i!}h_{X}(t).
\end{align}
In a similar manner, the hazard rate of $Y_{U_n}$ at time point $t$, is derived as
\begin{equation}\label{hrur2}
h_{Y_{U_n}}(t)=\displaystyle\frac{1}{(n-1)! \sum_{i=0}^{n-1}(-\ln(\bar{F}_Y(t)))^{i-n+1}/i!}h_{Y}(t).
\end{equation}

Next, we present a result on the preservation property of relative hazard rate order
under distribution of upper records. \medskip

\begin{theorem}\label{thm1}
Let $X\succeq_{st}Y$ and also let $m\geq n \in \mathbb{N}$. Then, $X\preceq_{c}Y$ implies $X_{U_m}\preceq_{c} Y_{U_n}$.
\end{theorem}

\begin{proof}
We first prove that $X_{U_m}\preceq_{c} X_{U_n}$ for every $n\in \mathbb{N}$ and for all $m=n,n+1,\ldots$. Then, we
establish that $X_{U_n}\preceq_{c} Y_{U_n}$. From transitivity property of the relative hazard rate order, one concludes
that $X_{U_m}\preceq_{c} Y_{U_n}$. For $m\geq n \in \mathbb{N}$, let us denote
$$K(u):=\frac{\sum_{i=0}^{n-1} u^{i-n+1}/i!}{\sum_{i=0}^{m-1} u^{i-m+1}/i!}.$$
By using (\ref{hrur1}), one can then write
$$\frac{h_{X_{U_m}}(t)}{h_{X_{U_n}}(t)}=\frac{(n-1)!}{(m-1)!}K(-\ln(\bar{F}_X(t))),~~t\geq 0.$$
It suffices to show that $K(u)$ increases in $u\geq 0$. One can get
\begin{align*}
K(u)&=\frac{\sum_{i=0}^{n-1} u^{i+m-n}/i!}{\sum_{i=0}^{m-1} u^{i}/i!}\\
&=\frac{\sum_{i=m-n}^{m-1} u^{i}/(i-(m-n))!}{\sum_{i=0}^{m-1} u^{i}/i!}.
\end{align*}
Note that $K(u)$ in increasing in $u\geq 0$, if and only if,
$$Q^{\ast}(j,u):=\sum_{i=0}^{m-1} Q(j,i)W(i,u)~~\textmd{is}~TP_2~\textmd{in}~(j,u)\in \{1,2\}\times [0,+\infty),$$
where $Q(j,i):=I[i\geq 0]/i!$ for $j=1$ and $Q(j,i):=I[i\geq m-n]/(i-(m-n))!$, for $j=2;$ and also $W(i,u):=u^i$. It is plain
to see that $Q(j,i)$ is $TP_2$ in $(j,i)\in \{1,2\}\times \{0,1,\ldots,m-1\}$ and $W(i,u)$ is also $TP_2$ in $(i,u)\in \{0,1,\ldots,m-1\}\times [0,+\infty)$. Thus, by using general composition theorem of Karlin \cite{Karlin1968}, $Q^{\ast}(j,u)$ is $TP_2$ in $(j,u)\in \{1,2\} \times [0,+\infty).$ Hence, we established that $X_{U_m}\preceq_{c} X_{U_n}$ for every $m>n\in \mathbb{N}$. \medskip

Now, we prove that $X_{U_n}\preceq_{c} Y_{U_n}$, for all $n\in \mathbb{N}.$ From Eq. (\ref{hrur1}) and Eq. (\ref{hrur2}), one has the following
$$\frac{h_{X_{U_n}}(t)}{h_{Y_{U_n}}(t)}=\frac{h_{X}(t)}{h_{Y}(t)}\cdot\frac{\sum_{i=0}^{n-1} (-\ln(\bar{F}_{Y}(t)))^{i-n+1} /i!}
{\sum_{i=0}^{n-1} (-\ln(\bar{F}_{X}(t)))^{i-n+1} /i!}.$$
Since $X \preceq_{c} Y$, thus $h_X(t)/h_Y(t)$ is increasing in $t\geq 0$. Thus, it is sufficient to prove that
$$\frac{\sum_{i=0}^{n-1} (-\ln(\bar{F}_{Y}(t)))^{i-n+1} /i!}
{\sum_{i=0}^{n-1} (-\ln(\bar{F}_{X}(t)))^{i-n+1} /i!}~~\textmd{is increasing in} ~~t\geq 0.$$
Now, we can write
\begin{align*}
\frac{\sum_{i=0}^{n-1} (-\ln(\bar{F}_{Y}(t)))^{i-n+1} /i!}{\sum_{i=0}^{n-1} (-\ln(\bar{F}_{X}(t)))^{i-n+1} /i!}&=
\sum_{i=0}^{n-1} \left( \frac{-\ln(\bar{F}_{Y}(t))}{-\ln(\bar{F}_{X}(t))} \right)^{i-n+1}\cdot \frac{(-\ln(\bar{F}_{X}(t)))^{i-n+1} /i!}
{\sum_{i=0}^{n-1} (-\ln(\bar{F}_{X}(t)))^{i-n+1} /i!}\\
&=E[\Phi^{\ast}(I^{\ast}(t),t)],
\end{align*}
where $\Phi^{\ast}(i,t):=\left( \frac{-\ln(\bar{F}_{Y}(t))}{-\ln(\bar{F}_{X}(t))} \right)^{i-n+1}$ and $I^{\ast}(t)$ is a discrete rv
with pmf
$$f^{\ast}(i|t)=\frac{(-\ln(\bar{F}_{X}(t)))^{i-n+1} /i!}
{\sum_{i=0}^{n-1} (-\ln(\bar{F}_{X}(t)))^{i-n+1} /i!}, i=0,1,\ldots,n-1.$$
Since $X\preceq_{c} Y,$ thus $\frac{h_Y(t)}{h_X(t)}$ is decreasing in $t\geq 0.$ Hence,
from Lemma 2.1 of Khaledi et al. \cite{Khaledi2009}, $(-\ln(\bar{F}_{Y}(t)))/(-\ln(\bar{F}_{X}(t)))$ is decreasing in
$t\geq 0$. Therefore, since for every $i=0,1,\ldots, n-1$, $i-n+1<0$, thus $\Phi^{\ast}(i,t)$ is increasing in $t\geq 0$.
On the other hand, from assumption we have $X \succeq_{st} Y$, and thus, for all $t\geq 0$, it holds that $-\ln(\bar{F}_{Y}(t))\geq -\ln(\bar{F}_{X}(t)).$
As a result, $\Phi^{\ast}(i,t)$ is an increasing function in $i=0,1,\ldots,n-1$, for every $t\geq 0.$ Further, it can be shown readily that $f^{\ast}(i|t)$ is $TP_2$ in $(i,t) \in \{1,2\} \times [0,+\infty)$, which further implies that $I^{\ast}(t_1) \preceq_{st} I^{\ast}(t_2),$
for all $t_1\leq t_2$. Now, by using Lemma 2.2(i) of Misra and van der Meulen \cite{Misra2003}, $E[\Phi^{\ast}(I^{\ast}(t),t)]$ is increasing in $t\geq 0$. Now, the result is proved.
\end{proof}

Stochastic comparison of convolutions of exponential random variables has been an important subject of research
in literature (see, e.g., Bon and P\~{a}lt\~{a}nea \cite{Bon1999} and Kochar and Ma \cite{Kochar1999}). However, relative
ordering of convolutions of exponential random variables has not been investigated in literature thus far. In the next example, we make the
relative hazard rate ordering of convolutions of i.i.d. exponential random variables according to the relative hazard rate order. Note that a
non-negative rv $T$ is said to have Erlang distribution with parameters $(n,\lambda)$ whenever it has pdf
$f_T(t)=\frac{\lambda(\lambda t)^{n-1} e^{-\lambda t}}{(n-1)!},$ for all $t\geq 0$, where $n=1,2,\ldots$ and $\lambda>0$. \medskip

\begin{example}\label{example1}
Let us suppose that $X_1,X_2,\ldots$ and $Y_1,Y_2,\ldots$ are two i.i.d. sequences of exponential random variables with means
$1/\lambda_1$ and $1/\lambda_2$, respectively, such that $\lambda_1 <\lambda_2$. From Eq. (\ref{eq2}), $X_{U_m}$ and $Y_{U_n}$ have pdfs
$$f_{X_{U_m}}(x)=\frac{\lambda_1(\lambda_1 x)^{m-1}e^{-\lambda_1 x}}{(m-1)!}~~\textmd{and}~~
f_{Y_{U_n}}(y)=\frac{\lambda_2(\lambda_2 y)^{n-1}e^{-\lambda_2 y}}{(n-1)!},$$
which are the pdfs of two Erlang distributions with parameters $(m,\lambda_1)$ and $(n,\lambda_2),$ respectively. On the other hand, it is a well-known fact that partial sums of i.i.d. exponential random variables $X_1,\ldots,X_m$ and $Y_1,\ldots,Y_n$ follow Erlang distribution with parameters $(m,\lambda_1)$ and $(n,\lambda_2)$, respectively. Therefore, $X_{U_m}$ and $Y_{U_n}$ are equal in distribution with $\sum_{i=1}^m X_i$ and $\sum_{i=1}^n Y_i$, respectively. From Eq. (\ref{hrur1}) and Eq. (\ref{hrur2}), one gets
$$h_{\sum_{i=1}^m X_i}(t)=\frac{\lambda_1}{(m-1)! \sum_{i=0}^{m-1} (\lambda_1 t)^{i-m+1}/i!},~~
h_{\sum_{i=1}^n Y_i}(t)=\frac{\lambda_2}{(n-1)! \sum_{i=0}^{n-1} (\lambda_2 t)^{i-n+1}/i!}.$$
Since $\lambda_1 <\lambda_2$, thus $X_1 \succeq_{st} Y_1$. Thus, for every
$m \geq n \in \mathbb{N}$, by using Theorem \ref{thm1}, one concludes that $\sum_{i=1}^m X_i \preceq_{c} \sum_{i=1}^n Y_i.$ In Figure \ref{fig1},
we have plotted the hazard rate ratio of convolutions of exponential random variables, which is an increasing function.
\end{example}

\begin{figure}
\begin{center}
\includegraphics[trim= 1cm 1cm 1cm 1cm,height=6cm, angle=0]{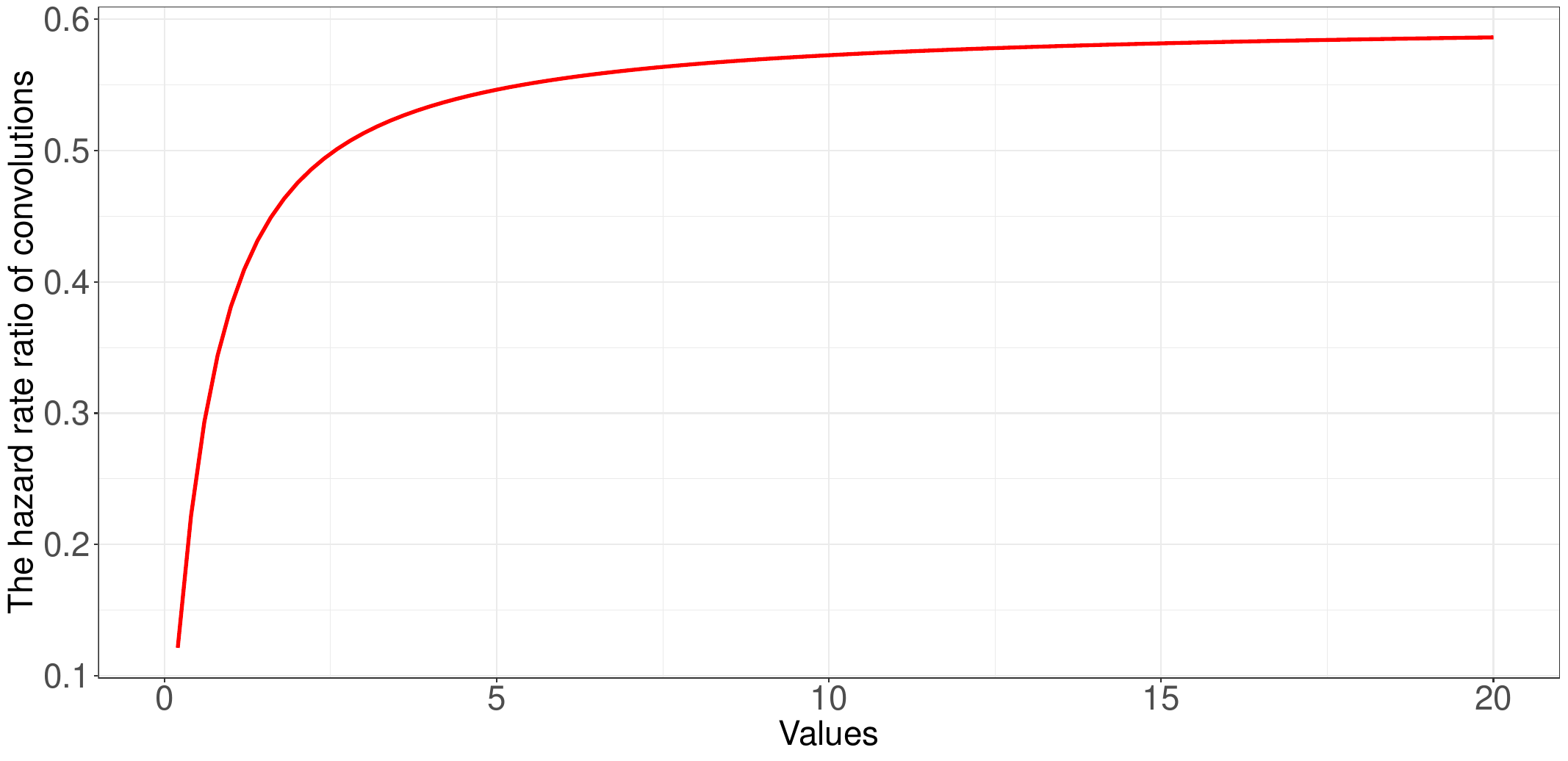}
\end{center}
\caption{The plot of the function $\frac{h_{\sum_{i=1}^m X_i}(t)}{h_{\sum_{i=1}^n Y_i}(t)}$ in Example \ref{example1} for $m=3, n=2,\lambda_1=3$ and $\lambda_2=5$.    \label{fig1}}
\end{figure}

In Theorem \ref{thm1} there is an order condition that $X \succeq_{st} Y$, However, this is a strong condition that may not be satisfied. In the residual
part of this section, we want to relax the condition $X \succeq_{st} Y$ to obtain the preservation property of the relative hazard rate order under upper records. The following lemma is useful to prove such result.

\begin{lemma}\label{lemm1}
Let $\xi_n(u)=(n-1)! \sum_{i=0}^{n-1} \frac{(-\ln(u))^{i-n+1}}{i!},$ for $u \in (0,1)$. Then, $\frac{u\xi'_n(u)}{\xi_n(u)}$ is a non-negative increasing
function of $u \in (0,1)$ for every $n=1,2,\ldots$.
\end{lemma}

\begin{proof}
Let us suppose that $I(u)$ is a discrete random variable with the following probability mass function
$$f(i\mid u):=\frac{(-\ln(u))^{i-n+1}/i!}{\sum_{i=0}^{n-1} (-\ln(u))^{i-n+1} /i!},~i=0,1,\ldots,n-1.$$
Then, one has the following
\begin{align*}
\frac{u \xi'_{n}(u)}{\xi_n(u)}&=\frac{\sum_{i=0}^{n-1} (n-i-1)(-\ln(u))^{i-n}/i!}{\sum_{i=0}^{n-1}(-\ln(u))^{i-n+1}/i!}\\
&=\sum_{i=0}^{n-1} \left(\frac{n-i-1}{-\ln(u)}\right)\cdot\frac{(-\ln(u))^{i-n+1}/i!}{\sum_{i=0}^{n-1} (-\ln(u))^{i-n+1} /i!}\\
&=\sum_{i=0}^{n-1} \Psi(i,u)\cdot f(i\mid u)\\
&=E[\Psi(I(u),u)],
\end{align*}
where $\Psi(i,u):=\frac{n-i-1}{-\ln(u)}$ is a decreasing function of $i=1,2,\ldots,n-1$ and it is an increasing function of $u\in (0,1)$. Note that $f(i\mid u)$ is an $RR_2$ function in $(i,u)\in \{1,2\}\times (0,1)$. Hence, $I(u_1)\geq_{st} I(u_2),$ for every $u_1 \leq u_2 \in (0,1)$, i.e., $I(u)$ is stochastically decreasing in $u\in (0,1)$. Thus, by using Lemma 2.2(i) of Misra and van der Meulen \cite{Misra2003}, $E[\Psi(I(u),u)]$ is an increasing function of $u$. This completes the proof of lemma.
\end{proof}

The limiting behaviour of hazard rates ratio is an important problem in reliability analysis (see, e.g., Navarro and Sarabia \cite{Navarro2024}). Let us define
$$c_0=\lim_{t\rightarrow 0^{+}} \frac{h_{X}(t)}{h_Y(t)},~~\textmd{and}~~c_1=\lim_{t\rightarrow +\infty} \frac{h_{X}(t)}{h_Y(t)}.$$
In what follows, we assume that $c_0$ and $c_1$ are positive and finite. The following theorem is the main result of this section.\medskip

\begin{theorem}\label{thm2}
Let $\psi_{n}(u)=\frac{u\xi'_{n}(u)}{\xi_n(u)}$  where $\xi_n(u)=(n-1)! \sum_{i=0}^{n-1} \frac{(-\ln(u))^{i-n+1}}{i!},$ for $u \in (0,1)$. If  $\sup_{0<u<1} \frac{\psi_{n}(u)}{\psi_{m}(u^{c_1})} \leq c_0$ then,
$X\preceq_{c}Y$ implies $X_{U_m} \preceq_{c} Y_{U_n}.$
\end{theorem}

\begin{proof}
From Eq. (\ref{hrur1}) and Eq. (\ref{hrur2}), for all $t\geq0,$ one has the following
$$h_{X_{U_m}}(t)=\frac{h_X(t)}{\xi_m(\bar{F}_X(t))},~~h_{Y_{U_n}}(t)=\frac{h_Y(t)}{\xi_n(\bar{F}_Y(t))}.$$
Thus, it is evident that
$$\frac{h_{X_{U_m}}(t)}{h_{Y_{U_n}}(t)}=\frac{h_X(t)}{h_{Y}(t)}\cdot\frac{\xi_n(\bar{F}_Y(t))}{\xi_m(\bar{F}_X(t))}.$$
Since from assumption $X\preceq_{c}Y,$ thus it is enough to
show that $\frac{\xi_n(\bar{F}_Y(t))}{\xi_m(\bar{F}_X(t))}$ is increasing in $t\geq 0$, or equivalently demonstrate that
\begin{equation}\label{eq8}
\frac{\partial}{\partial t}\ln\left( \frac{\xi_n(\bar{F}_Y(t))}{\xi_m(\bar{F}_X(t))}\right) \geq 0,~~\textmd{for every }~t\geq0.
\end{equation}
We have

\begin{align*}
\frac{\partial}{\partial t}\ln\left( \frac{\xi_n(\bar{F}_Y(t))}{\xi_m(\bar{F}_X(t))}\right)&=-\frac{f_Y(t)\xi'_n(\bar{F}_Y(t))}{\xi_n(\bar{F}_Y(t))}
+\frac{f_{X}(t)\xi'_m(\bar{F}_X(t))}{\xi_m(\bar{F}_X(t))}\\
&=-\frac{f_Y(t)}{\bar{F}_Y(t)}\frac{\bar{F}_Y(t)\xi'_n(\bar{F}_Y(t))}{\xi_n(\bar{F}_Y(t))}
+\frac{f_{X}(t)}{\bar{F}_X(t)}\frac{\bar{F}_{X}(t)\xi'_m(\bar{F}_X(t))}{\xi_m(\bar{F}_X(t))}\\
&=h_X(t)\cdot\frac{\bar{F}_{X}(t)\xi'_m(\bar{F}_X(t))}{\xi_m(\bar{F}_X(t))}-h_Y(t)\cdot\frac{\bar{F}_Y(t)\xi'_n(\bar{F}_Y(t))}{\xi_n(\bar{F}_Y(t))}.
\end{align*}

Note that since $\frac{h_X(t)}{h_{Y}(t)}$ is increasing in $t\geq 0$, thus $\frac{h_X(t)}{h_{Y}(t)}\geq c_0$, for all $t\geq 0,$ and  $\frac{h_X(t)}{h_{Y}(t)}\leq c_1$, for all $t\geq 0,$ which further imply that $h_X(t)\geq c_0 h_Y(t),$ for all $t\geq 0$ and
that $\bar{F}_X(t)\geq \bar{F}^{c_1}_{Y}(t),$ for all $t\geq 0,$ respectively. The second implication is due to the recursive formulas
$\bar{F}_X(t)=e^{-\int_0^t h_X(x) dx}$ and $\bar{F}_Y(t)=e^{-\int_0^t h_Y(x) dx}$ so that if $h_X(t) \leq c_1 h_Y(t),$ for all $t\geq 0$,
then one concludes $\bar{F}_X(t)\geq \bar{F}^{c_1}_{Y}(t),$ for all $t\geq 0.$ Now, one has the following

\begin{align*}
\frac{\partial}{\partial t}\ln\left( \frac{\xi_n(\bar{F}_Y(t))}{\xi_m(\bar{F}_X(t))}\right)&=h_X(t)\cdot\frac{\bar{F}_{X}(t)\xi'_m(\bar{F}_X(t))}
{\xi_m(\bar{F}_X(t))}-h_Y(t)\cdot\frac{\bar{F}_Y(t)\xi'_n(\bar{F}_Y(t))}{\xi_n(\bar{F}_Y(t))}\\
&\geq h_{Y}(t)\cdot\left(c_0\frac{\bar{F}_{X}(t)\xi'_m(\bar{F}_X(t))}{\xi_m(\bar{F}_X(t))}-\frac{\bar{F}_Y(t)\xi'_n(\bar{F}_Y(t))}{\xi_n(\bar{F}_Y(t))}\right)\\
&\geq h_{Y}(t)\cdot\left(c_0\frac{\bar{F}^{c_1}_{Y}(t)\xi'_m(\bar{F}^{c_1}_{Y}(t))}{\xi_m(\bar{F}^{c_1}_{Y}(t))}
-\frac{\bar{F}_Y(t)\xi'_n(\bar{F}_Y(t))}{\xi_n(\bar{F}_Y(t))}\right)\\
&=h_{Y}(t)\cdot\left(c_0\psi_m(\bar{F}^{c_1}_{Y}(t))-\psi_n(\bar{F}_Y(t))\right)
\end{align*}

where the last inequality holds true because $\bar{F}_X(t)\geq \bar{F}^{c_1}_{Y}(t),$ for all $t\geq 0$ and, further, since
from Lemma \ref{lemm1}, $\frac{u\xi'_{n}(u)}{\xi_n(u)}$ is increasing in $u\in (0,1).$  Now, one can see that (\ref{eq8}) is
satisfied if $c_0\psi_m(\bar{F}^{c_1}_{Y}(t))\geq \psi_n(\bar{F}_Y(t)),$ for all $t\geq 0$. Since $\psi_m$ is a non-negative function, and
moreover, since $\sup_{t\geq 0} \frac{\psi_{n}(\bar{F}_Y(t))}{\psi_{m}(\bar{F}^{c_1}_{Y}(t))}=\sup_{0<u<1} \frac{\psi_{n}(u)}{\psi_{m}(u^{c_1})}$ thus
the recent inequality holds if, and only if,
$$\sup_{0<u<1} \frac{\psi_{n}(u)}{\psi_{m}(u^{c_1})} \leq c_0.$$
\end{proof}

\begin{remark}\label{rem1}
It is to be mentioned that the result of Theorem \ref{thm2} in situations where $X$ and $Y$ are equal in distribution, could also be achieved by using Theorem \ref{thm1}, since the assumption $X\succeq_{st}Y$ is satisfied when $X$ and $Y$ have a common
distribution. It is remarkable in the context of Theorem \ref{thm2}, that if $X$ and $Y$ are equal in distribution, then $c_0=c_1=1$. Note that
$\psi_n(u)$ is increasing in $n \geq 1$, if and only if for every $m\geq n \in \mathbb{N}$,
$$\psi_m(u):=\frac{u\xi'_{m}(u)}{\xi_m(u)}\geq \frac{u\xi'_{n}(u)}{\xi_n(u)}=\psi_{n}(u),~~\textmd{for all}~u\in (0,1).$$
The above inequality is also equivalent to
$$\frac{\xi'_{m}(u)}{\xi_m(u)}\geq \frac{\xi'_{n}(u)}{\xi_n(u)},~~\textmd{for all}~u\in (0,1),$$
which means that $\frac{\xi_m(u)}{\xi_n(u)}$ is increasing in $u \in (0,1).$ We can write
\begin{align*}
\frac{\xi_m(u)}{\xi_n(u)}&=\frac{(m-1)! \sum_{i=0}^{m-1} (-\ln(u))^{i-m+1} /i!}{(n-1)! \sum_{i=0}^{n-1} (-\ln(u))^{i-n+1} /i!}\\
&=\frac{(m-1)!}{(n-1)!}\cdot \frac{\sum_{i=0}^{m-1} (-\ln(u))^{i} /i!}{\sum_{i=0}^{n-1} (-\ln(u))^{i-n+m} /i!}\\
&=\frac{(m-1)!}{(n-1)!}\cdot \frac{\sum_{i=0}^{m-1} (-\ln(u))^{i} /i!}{\sum_{i=m-n}^{m-1} (-\ln(u))^{i} /(i-(m-n))!},
\end{align*}
in which $\frac{\sum_{i=0}^{m-1} (-\ln(u))^{i} /i!}{\sum_{i=m-n}^{m-1} (-\ln(u))^{i} /(i-(m-n))!}$ is increasing in $u \in (0,1)$, if and only if,
$$q^{\star}(j,u):=\sum_{i=0}^{m-1} q(j,i)w(i,u)~\textmd{is}~ TP_2 ~\textmd{in} ~(j,u)\in \{1,2\} \times (0,1)$$
where $q(j,i):=I[i\geq 0]/i!$ for $j=2$ and
$q(j,i):=I[i\geq m-n]/(i-(m-n))!$ for $j=1$, and $w(i,u):=(-\ln(u))^i$. Since $q(j,i)$ is $RR_2$ in $(j,i)\in \{1,2\} \times \{0,1,\ldots, m-1\}$
and $w(i,u)$ is also $RR_2$ in $(i,u)\in \{1,2\} \times (0,1),$ thus the required result follows from general composition theorem of Karlin
\cite{Karlin1968}. Now, since $c_0=c_1=1$ as discussed before, thus for every $m=n,n+1,\ldots$ where $n\in \mathbb{N}$ one has
\begin{align*}
\sup_{0\leq u\leq 1} \frac{\psi_{n}(u)}{\psi_{m}(u^{c_1})}&=\sup_{0\leq u\leq 1} \frac{\psi_{n}(u)}{\psi_{m}(u)} \\
&\leq 1=c_0,
\end{align*}
where the last inequality holds because $\psi_n(u) \leq \psi_m(u),$ for every $u\in [0,1].$ Therefore, from Theorem \ref{thm2}, we deduced that if $m\geq n \in \mathbb{N},$ then $X_{U_m} \preceq_{c} X_{U_n}$.
\end{remark}

The following example presents an example where Theorem \ref{thm2} is applicable. However, we can not deduce the result using Theorem \ref{thm1}
because $X \nsucceq_{st} Y$. \medskip

\begin{example}\label{exa7}
Let us consider $X$ and $Y$ as two random lifetimes with Lomax distributions having sfs $\bar{F}_X(t)=\frac{1}{(1+t)^3}$ and $\bar{F}_Y(t)=\frac{9}{(3+4t)^2}.$ The hr functions of $X$ and $Y$ are then calculated as $h_X(t)=\frac{3}{1+t}$ and $h_Y(t)=\frac{8}{3+4t}$, which yields $\frac{h_X(t)}{h_Y(t)}=\frac{9+12t}{8+8t}.$ The function $\frac{h_X(t)}{h_Y(t)}$ is increasing in $t\geq 0$.
Thus, according to definition, $X\preceq_{c} Y$. It is seen that $c_0=\lim_{t\rightarrow 0^{+}} \frac{h_X(t)}{h_Y(t)}=\frac{9}{8}$ and $c_1=\lim_{t\rightarrow +\infty} \frac{h_X(t)}{h_Y(t)}=\frac{3}{2}$. It is trivial that, for all $t\geq 0$, $h_X(t) \geq h_Y(t)$ which means $X\preceq_{hr}Y$, and consequently,
 $X\preceq_{st} Y$. Therefore, the condition of Theorem \ref{thm1} does not hold. We want to compare the second upper record from a sequence of i.i.d.
  random lifetimes adopted from $F_X$ to the third upper record from a sequence of i.i.d.
  random lifetimes following $F_Y$. Hence, $n=2$ and $m=3$. We show using the setting of Theorem \ref{thm2} that $X_{U_m}\preceq_{c} Y_{U_n}.$
For $u\in [0,1]$, we observe that
$$\xi_n(u)=\frac{1}{-\ln(u)}+1, ~\textmd{and}~\xi_m(u)=1+\frac{2}{-\ln(u)}+\frac{2}{(-\ln(u))^2},$$
from which one further obtains
$$\xi'_n(u)=\frac{1}{u(-\ln(u))^2},~\textmd{and}~\xi'_m(u)=\frac{2}{u(-\ln(u))^2}+\frac{4}{u(-\ln(u))^3}.$$
After some routine calculation, we derive
$$\psi_n(u)=\frac{u\xi'_n(u)}{\xi_n(u)}=\frac{1}{(-\ln(u))\cdot(-\ln(u)+1)}$$
and, similarly,
$$\psi_m(u)=\frac{u\xi'_m(u)}{\xi_m(u)}=\frac{-2\ln(u)+4}{(-\ln(u))\cdot((-\ln(u))^2-2\ln(u)+2)}.$$
Therefore,
$$\frac{\psi_n(u)}{\psi_m(u^{c_1})}=\frac{\psi_n(u)}{\psi_m(u^{\frac{3}{2}})}=\frac{27(-\ln(u))^3 +36(-\ln(u))^2 +24(-\ln(u))}{(-\ln(u))\cdot(1-\ln(u))\cdot(32-24\ln(u))}.$$
It can be checked using website \emph{https://www.dcode.fr/maximum-function} that $\sup_{0\leq u\leq 1} \frac{\psi_{n}(u)}{\psi_{m}(u^{c_1})}=\frac{9}{8}$. One can also see Figure \ref{fig2}. Hence, the condition of Theorem \ref{thm2} is satisfied and hence the required result follows.
\end{example}
In Example \ref{exa7}, one can further obtain
$$h_{X_{U_m}}(t)=\frac{h_X(t)}{\xi_m(\bar{F}_X(t))}=\frac{27\ln^2(1+t)}{(1+t)\cdot(2+6\ln(1+t)+9\ln^2(1+t))},$$
and
$$h_{Y_{U_n}}(t)=\frac{h_Y(t)}{\xi_n(\bar{F}_Y(t))}=\frac{8(-\ln(9)+2\ln(3+4t))}{(3+4t)\cdot(1-\ln(9)+2\ln(3+4t))}.$$
In Figure \ref{fig3} we plot the graph of $\frac{h_{X_{U_m}}(t)}{h_{Y_{U_n}}(t)}$ to exhibit that it is an increasing function
in $t$.

\begin{figure}
\begin{center}
\includegraphics[trim= 1cm 1cm 1cm 1cm,height=6cm, angle=0]{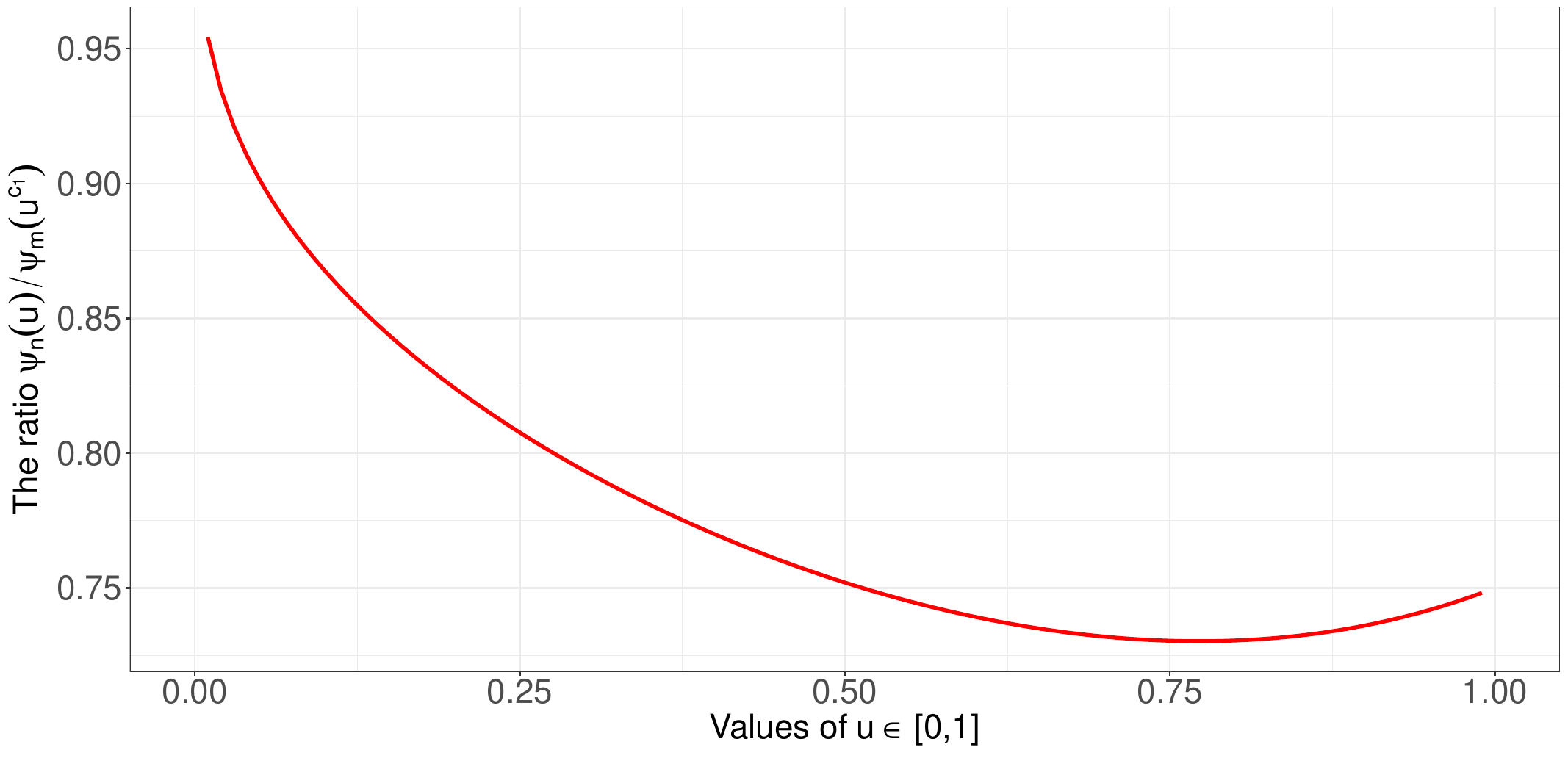}
\end{center}
\caption{The graph of the function $\frac{\psi_n(u)}{\psi_m(u^{c_1})}$ in Example \ref{exa7}.    \label{fig2}}
\end{figure}

\begin{figure}
\begin{center}
\includegraphics[trim= 1cm 1cm 1cm 1cm,height=6cm, angle=0]{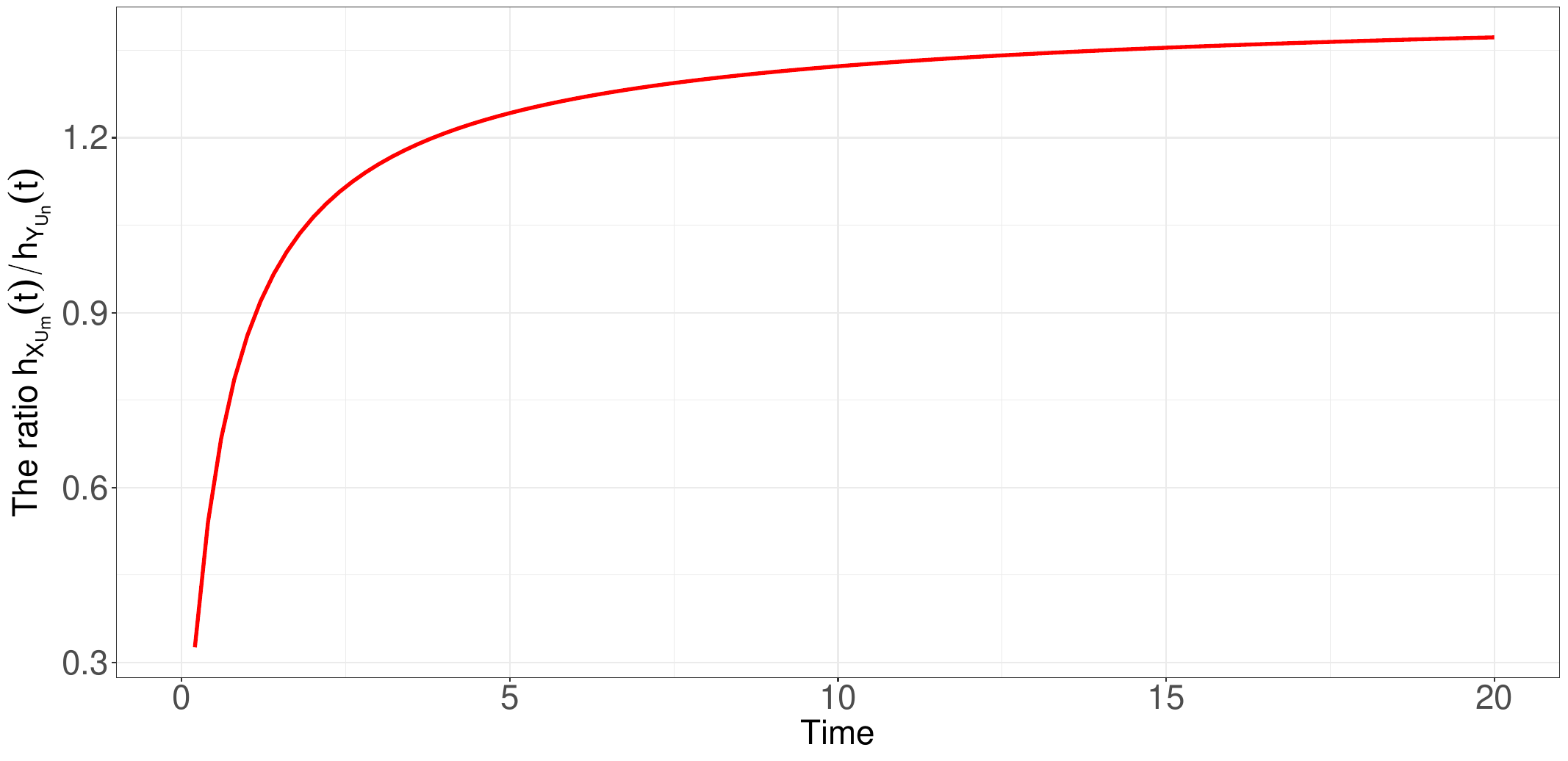}
\end{center}
\caption{The graph of $\frac{h_{X_{U_m}}(t)}{h_{Y_{U_n}}(t)}$ in Example \ref{exa7}.    \label{fig3}}
\end{figure}

\section{Relative comparison of lower record values}
\label{sec:4}
In this section, we consider the lower records $X_{L_m}$ and $Y_{L_n}$ from two sequences of rvs
$\{X_i: i=1,2,\ldots\}$ and $\{Y_i: i=1,2,\ldots\}$, each containing i.i.d.
observations. The rhr of $X_{L_m}$ is obtained as
\begin{align}\label{rhrlwr1}
r_{X_{L_m}}(t)&=\frac{f_{X_{L_m}}(t)}{F_{X_{L_m}}(t)}\nonumber \\
&=\displaystyle\frac{(-\ln(F_X(t)))^{m-1}}{(m-1)! \sum_{i=0}^{m-1}(-\ln(F_X(t)))^{i}/i!}r_{X}(t)\nonumber \\
&=\displaystyle\frac{1}{(m-1)! \sum_{i=0}^{m-1}(-\ln(F_X(t)))^{i-m+1}/i!}r_{X}(t).
\end{align}
Analogously, the rhr of $Y_{L_n}$ at time point $t$, is acquired as
\begin{equation}\label{rhrlwr2}
r_{Y_{L_n}}(t)=\displaystyle\frac{1}{(n-1)! \sum_{i=0}^{n-1}(-\ln(F_Y(t)))^{i-n+1}/i!}r_{Y}(t).
\end{equation}

Next, we present a result on the preservation property of relative reversed hazard rate order
under distribution of lower records. \medskip

\begin{theorem}\label{thm3}
Let $X\preceq_{st}Y$ and $m\geq n \in \mathbb{N}$. Then, $X\preceq_{b}Y$ yields $X_{L_m}\preceq_{b} Y_{L_n}$.
\end{theorem}

\begin{proof}
Firstly, we demonstrate that $X_{L_m}\preceq_{b} X_{L_n}$ holds true for $n\in \mathbb{N}$ and for every $m\geq n$ such that
$m\in \mathbb{N}$. Then, it will be proved that $X_{L_n}\preceq_{b} Y_{L_n}$. Due to the transitivity property of the relative
reversed hazard rate order, it is deduced that $X_{L_m}\preceq_{b} Y_{L_n}$. Keeping the notations of Theorem \ref{thm1} in mind,
and in view of (\ref{rhrlwr1}), we can write
$$\frac{r_{X_{L_m}}(t)}{r_{X_{L_n}}(t)}=\frac{(n-1)!}{(m-1)!}K(-\ln(F_X(t))),~~t\geq 0.$$
In the proof of Theorem \ref{thm1}, it was shown that $K(u)$ in increasing in $u\geq 0$.  Hence, since $-\ln(F_X(t))$ is decreasing in $t>0$, thus
$\frac{r_{X_{L_m}}(t)}{r_{X_{L_n}}(t)}$ is decreasing in $t>0$. Therefore, $X_{L_m}\preceq_{b} X_{L_n}$ for every $m=n,n+1,\ldots$. To finalize the
proof, we need to prove $X_{L_n}\preceq_{b} Y_{L_n}$, for every $n\in \mathbb{N}.$ In the spirit of Eq. (\ref{rhrlwr1}) and Eq. (\ref{rhrlwr2}), we
get
$$\frac{r_{X_{L_n}}(t)}{r_{Y_{L_n}}(t)}=\frac{r_{X}(t)}{r_{Y}(t)}\cdot\frac{\sum_{i=0}^{n-1} (-\ln(F_{Y}(t)))^{i-n+1} /i!}
{\sum_{i=0}^{n-1} (-\ln(F_{X}(t)))^{i-n+1} /i!}.$$
Following assumption, we have $X \preceq_{b} Y$. Hence $r_X(t)/r_Y(t)$ is decreasing in $t>0$. It thus suffices to establish that
$$\frac{\sum_{i=0}^{n-1} (-\ln(F_{Y}(t)))^{i-n+1} /i!}
{\sum_{i=0}^{n-1} (-\ln(F_{X}(t)))^{i-n+1} /i!}~~\textmd{is decreasing in} ~~t> 0.$$
Let us write
\begin{align*}
\frac{\sum_{i=0}^{n-1} (-\ln(F_{Y}(t)))^{i-n+1} /i!}{\sum_{i=0}^{n-1} (-\ln(F_{X}(t)))^{i-n+1} /i!}&=
\sum_{i=0}^{n-1} \left( \frac{-\ln(F_{Y}(t))}{-\ln(F_{X}(t))} \right)^{i-n+1}\cdot\frac{(-\ln(F_{X}(t)))^{i-n+1} /i!}
{\sum_{i=0}^{n-1} (-\ln(F_{X}(t)))^{i-n+1} /i!}\\
&=E[\Phi_{\ast}(I_{\ast}(t),t)],
\end{align*}
where $\Phi_{\ast}(i,t):=\left(\frac{-\ln(F_{Y}(t))}{-\ln(F_{X}(t))} \right)^{i-n+1}$ and $I_{\ast}(t)$ is an rv
with pmf
$$f_{\ast}(i|t)=\frac{(-\ln(F_{X}(t)))^{i-n+1} /i!}
{\sum_{i=0}^{n-1} (-\ln(F_{X}(t)))^{i-n+1} /i!}, i=0,1,\ldots,n-1.$$
Since $X\preceq_{b} Y,$ thus from Hazra and Misra \cite{Hazra2020}, $\frac{-\ln(F_{Y}(t))}{-\ln(F_{X}(t))}$ is increasing in
$t>0$. Thus, $\Phi_{\ast}(i,t)$ is decreasing in $t>0$. From assumption, it holds that $X \preceq_{st} Y$, and as a result, for all $t>0$, $-\ln(F_{Y}(t))\geq -\ln(F_{X}(t)).$
Therefore, $\Phi_{\ast}(i,t)$ is increasing in $i=0,1,\ldots,n-1$, for all$t>0.$ It can be seen readily that $f_{\ast}(i|t)$ is $RR_2$ in $(i,t) \in \{1,2\} \times (0,+\infty)$, which in turn yields $I_{\ast}(t_1) \succeq_{st} I_{\ast}(t_2),$
for all $t_1\leq t_2$. On applying Lemma 2.2(ii) of Misra and van der Meulen \cite{Misra2003}, $E[\Phi_{\ast}(I_{\ast}(t),t)]$ is decreasing in $t>0$. Now, the result is proved.
\end{proof}

The following example presents a situation where Theorem \ref{thm3} is applicable.

\begin{example}\label{exa777}
We suppose that $X$ and $Y$ have inverse Weibull distribution with cdfs $F_X(t)=\exp(-\frac{4}{t^2})$ and
$F_Y(t)=\exp(-\frac{5}{t^2}),$ where $t>0.$ We can get $r_X(t)=\frac{8}{t^3}$ and $r_Y(t)=\frac{10}{t^3}.$
Since $\frac{r_X(t)}{r_Y(t)}=\frac{4}{5},$ thus $X\preceq_{b}Y$. It is also trivial to see that $X\preceq_{st} Y$.
Let us choose $n=2$ and $m=3$. The result of Theorem \ref{thm3} implies that $X_{L_m}\preceq_{b} Y_{L_n}.$ One has the following
$$r_{X_{L_m}}(t)=\frac{r_X(t)}{\xi_m(F_X(t))}=\frac{64}{t^3(t^4 +4t^2+8)}~~\textmd{and}~~r_{Y_{L_n}}(t)=\frac{r_Y(t)}{\xi_n(F_Y(t))}
=\frac{50}{t^3(t^2 +5)}.$$
In Figure \ref{fig4} we plot the graph of $\frac{r_{X_{L_m}}(t)}{r_{Y_{L_n}}(t)}$ and it is shown that this function decreases in $t>0$
which validates the result of Theorem \ref{thm3}.
\end{example}

\begin{figure}
\begin{center}
\includegraphics[trim= 1cm 1cm 1cm 1cm,height=6cm, angle=0]{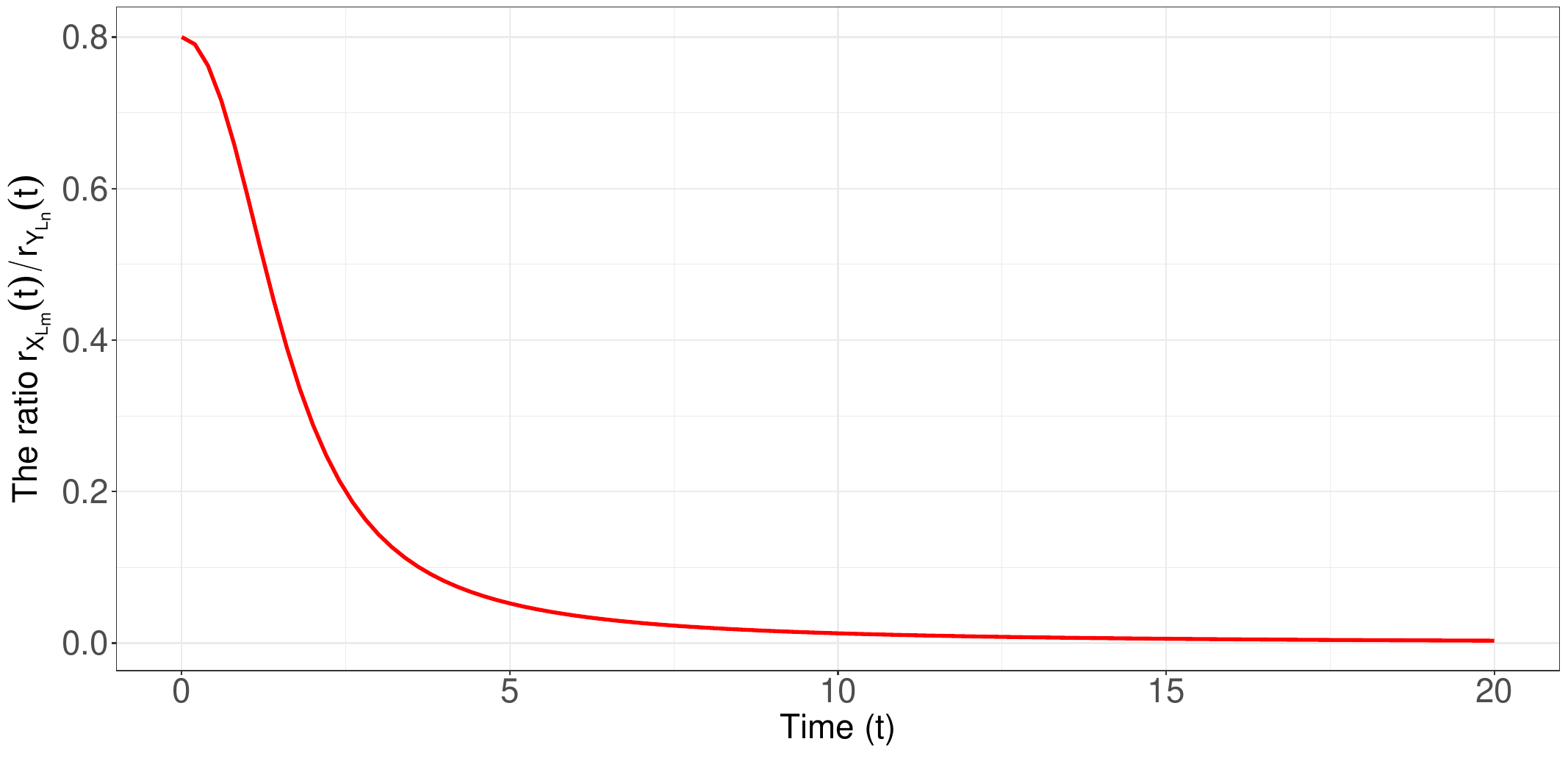}
\end{center}
\caption{The graph of $\frac{r_{X_{L_m}}(t)}{r_{Y_{L_n}}(t)}$ in Example \ref{exa777}.    \label{fig4}}
\end{figure}

The extremal behaviour of reversed hazard rates ratio is considered here. Let us define
$$c^{\ast}_0=\lim_{t\rightarrow 0^{+}} \frac{r_{X}(t)}{r_Y(t)},~~\textmd{and}~~c^{\ast}_1=\lim_{t\rightarrow +\infty} \frac{r_{X}(t)}{r_Y(t)}.$$
In the sequel, we assume that $c^{\ast}_0$ and $c^{\ast}_1$ are positive and finite. In the following theorem the condition of $X\preceq_{st}Y$ in Theorem \ref{thm3} is relaxed.

\begin{theorem}\label{thm4}
If  $\sup_{0<u<1} \frac{\psi_{n}(u^{\frac{1}{c_0^{\ast}}})}{\psi_{m}(u)} \leq c^{\ast}_1$ then,
$X\preceq_{b}Y$ implies $X_{L_m} \preceq_{b} Y_{L_n}.$
\end{theorem}

\begin{proof}
In view of Eq. (\ref{rhrlwr1}) and Eq. (\ref{rhrlwr2}), for all $t>0,$ we get
$$r_{X_{L_m}}(t)=\frac{r_X(t)}{\xi_m(F_X(t))},~~r_{Y_{L_n}}(t)=\frac{r_Y(t)}{\xi_n(F_Y(t))}.$$
Therefore,
$$\frac{r_{X_{L_m}}(t)}{r_{Y_{L_n}}(t)}=\frac{r_X(t)}{r_{Y}(t)}\cdot\frac{\xi_n(F_Y(t))}{\xi_m(F_X(t))}.$$
By assumption we have $X\preceq_{b}Y$ which means $\frac{r_X(t)}{r_{Y}(t)}$ is decreasing. Thus, we only need to show that $\frac{\xi_n(F_Y(t))}{\xi_m(F_X(t))}$ is decreasing in $t\geq 0$, or equivalently, we can prove that
\begin{equation}\label{eq88}
\frac{\partial}{\partial t}\ln\left( \frac{\xi_n(F_Y(t))}{\xi_m(F_X(t))}\right) \leq 0,~~\textmd{for all}~t>0.
\end{equation}
One has
\begin{align*}
\frac{\partial}{\partial t}\ln\left( \frac{\xi_n(F_Y(t))}{\xi_m(F_X(t))}\right)&=\frac{f_Y(t)\xi'_n(F_Y(t))}{\xi_n(F_Y(t))}
-\frac{f_{X}(t)\xi'_m(F_X(t))}{\xi_m(F_X(t))}\\
&=\frac{f_Y(t)}{F_Y(t)}\frac{F_Y(t)\xi'_n(F_Y(t))}{\xi_n(F_Y(t))}
-\frac{f_{X}(t)}{F_X(t)}\frac{F_{X}(t)\xi'_m(F_X(t))}{\xi_m(F_X(t))}\\
&=r_Y(t)\cdot\frac{F_{Y}(t)\xi'_n(F_Y(t))}{\xi_n(F_Y(t))}-r_X(t)\cdot\frac{F_X(t)\xi'_m(F_X(t))}{\xi_m(F_X(t))}.
\end{align*}

Since $\frac{r_X(t)}{r_{Y}(t)}$ is decreasing in $t>0$, thus $\frac{r_X(t)}{r_{Y}(t)}\leq c^{\ast}_0$, for all $t>0,$ and
$\frac{r_X(t)}{r_{Y}(t)}\geq c^{\ast}_1$, for all $t>0.$ Hence, $r_X(t)\leq c^{\ast}_0 r_Y(t),$ for all $t>0$ and
also $F_Y(t)\leq F^{\frac{1}{c_0^{\ast}}}_{X}(t),$ for all $t>0$ where the second inequality follows from the identities
$F_X(t)=e^{-\int_t^{+\infty} r_X(x) dx}$ and $F_Y(t)=e^{-\int_t^{+\infty} r_Y(x) dx}$. This is because if $r_X(t) \leq c^{\ast}_0 r_Y(t),$ for all $t>0$, then  $F_Y(t)\leq F^{\frac{1}{c_0^{\ast}}}_{X}(t),$ for all $t>0.$ Therefore,

\begin{align*}
\frac{\partial}{\partial t}\ln\left(\frac{\xi_n(F_Y(t))}{\xi_m(F_X(t))}\right)&=
r_Y(t)\cdot\frac{F_{Y}(t)\xi'_n(F_Y(t))}{\xi_n(F_Y(t))}-r_X(t).\frac{F_X(t)\xi'_m(F_X(t))}{\xi_m(F_X(t))}\\
&\leq r_{X}(t)\cdot\left(\frac{1}{c_1^{\ast}}\frac{F_{Y}(t)\xi'_n(F_Y(t))}{\xi_n(F_Y(t))}
-\frac{F_X(t)\xi'_m(F_X(t))}{\xi_m(F_X(t))}\right)\\
&\leq r_{X}(t)\cdot\left(\frac{1}{c_1^{\ast}}\frac{F^{\frac{1}{c_0^{\ast}}}_{X}(t)\xi'_n(F^{\frac{1}{c_0^{\ast}}}_{X}(t))}{\xi_n(F^{\frac{1}{c_0^{\ast}}}_{X}(t))}
-\frac{F_X(t)\xi'_m(F_X(t))}{\xi_m(F_X(t))}\right)\\
&=r_{X}(t)\cdot\left(\frac{1}{c_1^{\ast}}\psi_n(F^{\frac{1}{c_0^{\ast}}}_{X}(t))-\psi_m(F_X(t))\right)
\end{align*}
where the first inequality holds because $r_Y(t)\leq \frac{1}{c^{\ast}_1} r_X(t),$ for all $t>0$ and the last inequality is due to $F_Y(t)\leq F^{\frac{1}{c_0^{\ast}}}_{X}(t),$ for all $t>0,$ and
since by Lemma \ref{lemm1}, $\frac{u\xi'_{n}(u)}{\xi_n(u)}$ is increasing in $u\in (0,1).$  It is now realized that (\ref{eq88}) is
fulfilled if $c^{\ast}_1\psi_m (F_{X}(t))\geq \psi_n(F^{\frac{1}{c}}_X(t)),$ for all $t>0$. We know that $\psi_m$ is a non-negative function, and
further, it is trivial that $\sup_{t\geq 0} \frac{\psi_{n}(F^{\frac{1}{c_0 ^{\ast}}}_X(t))}{\psi_{m}(F_{X}(t))}=\sup_{0<u<1} \frac{\psi_{n}(u)}{\psi_{m}(u^{c_1})}$ thus
the recent inequality holds if, and only if,
$$\sup_{0<u<1} \frac{\psi_{n}(u^{\frac{1}{c_0 ^{\ast}}})}{\psi_{m}(u)} \leq c^{\ast}_1.$$
\end{proof}

It is remarkable here that when $X$ and $Y$ are equal in distribution
then, $X \preceq_{st} Y$ and further $X \preceq_{b} Y$. In the spirit of Remark \ref{rem1}, since
$\psi_n(u)$ is increasing in $u$ for every $u\in [0,1],$ and because $c_0^{\ast}=c_{1}^{\ast}=1$, thus  for every $m=n,n+1,\ldots,$
one observes that $$\sup_{0<u<1} \frac{\psi_{n}(u^{\frac{1}{c_0 ^{\ast}}})}{\psi_{m}(u)}=\sup_{0<u<1} \frac{\psi_{n}(u)}{\psi_{m}(u)}\leq 1.$$
Hence, using Theorem \ref{thm3}, one obtains $X_{L_m} \preceq_{b} X_{L_n}.$ This result could also be concluded by applying Theorem \ref{thm2}.\medskip

The next example illustrates a situation where Theorem \ref{thm3} is not applicable while the result
of Theorem \ref{thm4} works.

\begin{example}\label{exa7777}
Let $X$ and $Y$ be two random lifetimes with Inverse Weibull distributions having cdfs $F_X(t)=\exp(-\frac{4}{t^4})$ and
$F_Y(t)=\exp(-\frac{2}{t^4}).$ The rhr functions of $X$ and $Y$ are obtained as $r_X(t)=\frac{16}{t^5}$ and $r_Y(t)=\frac{8}{t^5}$, which further implies that $\frac{r_X(t)}{r_Y(t)}=2.$ The function $\frac{r_X(t)}{r_Y(t)}$ is, therefore, non-increasing in $t>0$.
By definition, $X\preceq_{b} Y$. We can see that $c^{\ast}_0=c_1^{\ast}=\lim_{t\rightarrow 0^{+}} \frac{r_X(t)}{r_Y(t)}=2$. For all $>0$, $r_X(t) \geq r_Y(t)$ i.e., $X\succeq_{rh}Y$ and moreover $X\succeq_{st} Y$. This means that the condition of Theorem \ref{thm3} is not satisfied.
Here, we take $n=2$ and $m=3$. By using Theorem \ref{thm4} we shall conclude that $X_{L_m}\preceq_{b} Y_{L_n}.$
For $u\in [0,1]$, in the sprit of the proof of Theorem \ref{thm2} in Example \ref{exa7}, one gets
$$\frac{\psi_n(u^{\frac{1}{c_0^{\ast}}})}{\psi_m(u)}=\frac{\psi_n(u^{\frac{1}{2}})}{\psi_m(u)}=\frac{(-\ln(u))^2 -2\ln(u)+2}{(4-2\ln(u)).(\frac{1}{2}-\frac{1}{4}\ln(u))}$$
We can check using website \emph{https://www.dcode.fr/maximum-function} that $\sup_{0\leq u\leq 1} \frac{\psi_{n}(\sqrt{u})}{\psi_{m}(u)}=2$ as is
shown in Figure \ref{fig22}. The condition of Theorem \ref{thm4} holds and the result follows.
\end{example}
In Example \ref{exa7777}, we can obtain the rhr functions of $X_{L_m}$ and $Y_{L_n}$ as follows:
$$r_{X_{L_m}}(t)=\frac{r_X(t)}{\xi_m(F_X(t))}=\frac{256}{t^4(16+8t^4+2t^8)},$$
and
$$r_{Y_{L_n}}(t)=\frac{r_Y(t)}{\xi_n(F_Y(t))}=\frac{16(t^4 +2}{16+8t^4+2t^8}.$$
In Figure \ref{fig33} we have plotted the graph of $\frac{r_{X_{L_m}}(t)}{r_{Y_{L_n}}(t)}$ to exhibit that it decreases
in $t$. \bigskip

\begin{figure}
\begin{center}
\includegraphics[trim= 1cm 1cm 1cm 1cm,height=6cm, angle=0]{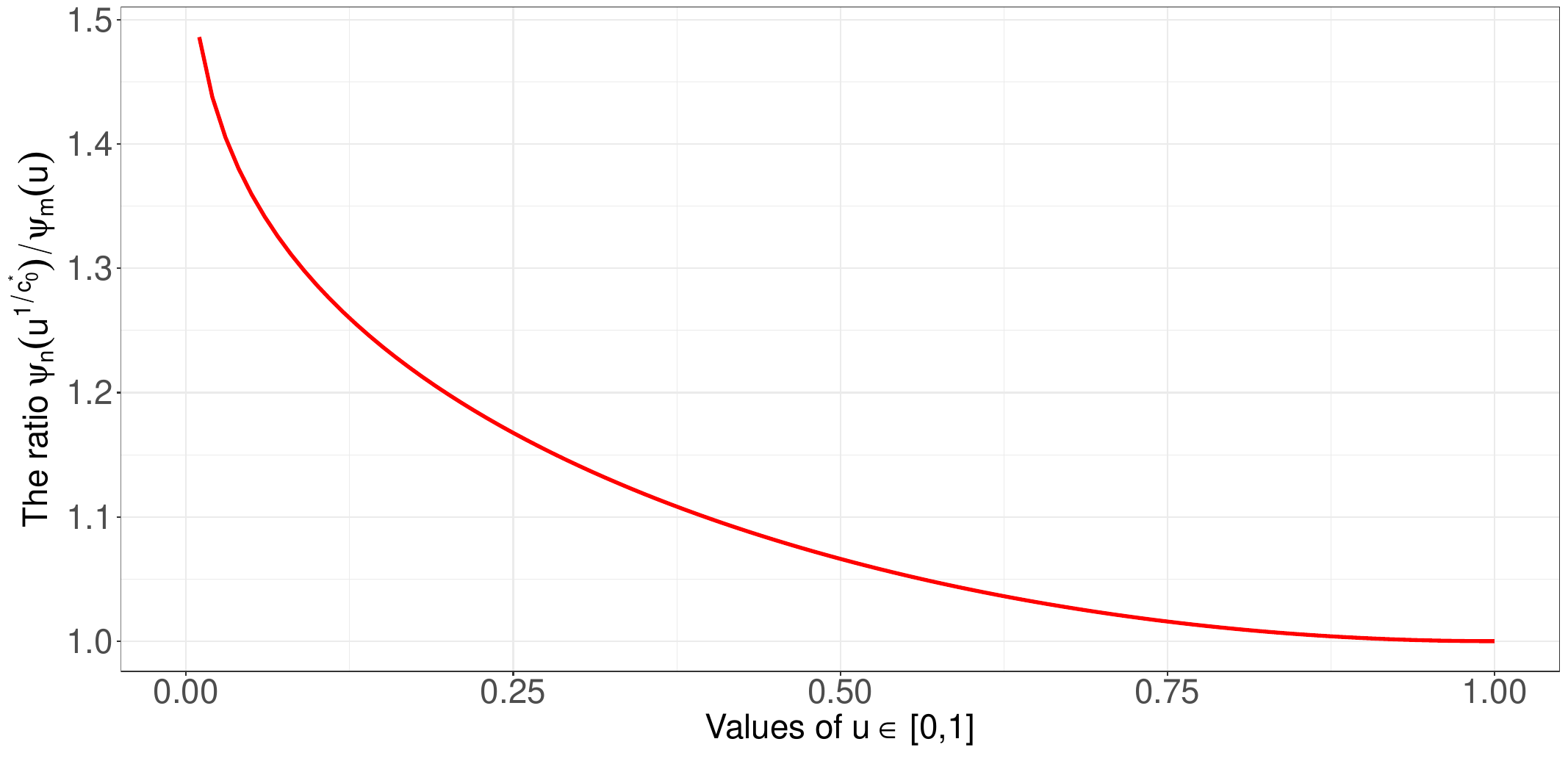}
\end{center}
\caption{The graph of the function $\frac{\psi_n(u^{\frac{1}{c_0^{\ast}}})}{\psi_m(u)}$ in Example \ref{exa7777}.    \label{fig22}}
\end{figure}

\begin{figure}
\begin{center}
\includegraphics[trim= 1cm 1cm 1cm 1cm,height=6cm, angle=0]{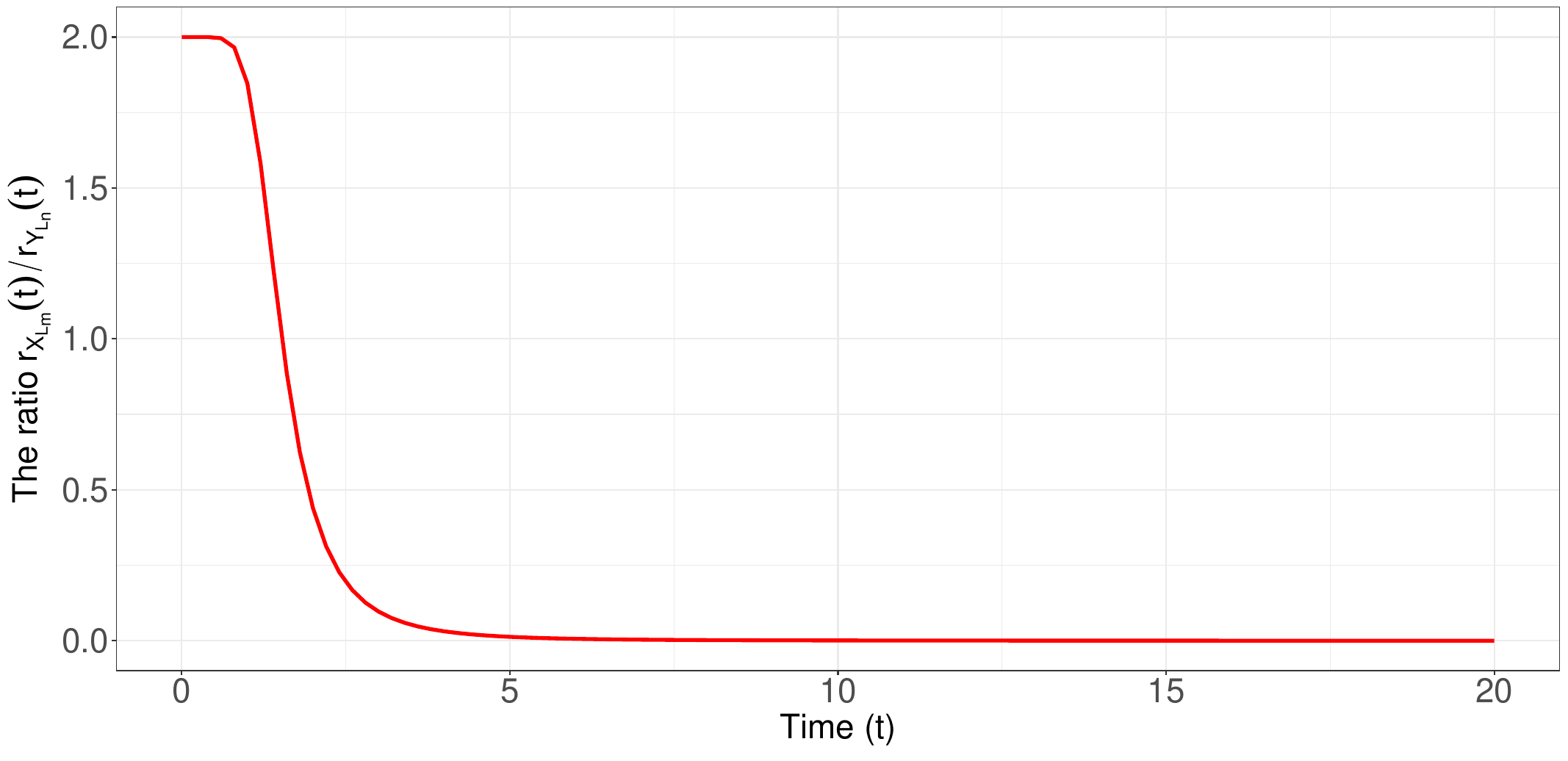}
\end{center}
\caption{The graph of $\frac{r_{X_{L_m}}(t)}{r_{Y_{L_n}}(t)}$ in Example \ref{exa7777}.    \label{fig33}}
\end{figure}

\section{Conclusion}
\label{sec:5}
In this paper, we first presented conditions under which the relative hazard rate order between $X$ and $Y$ with the respective distributions $F$ and $G$ is translated into the relative hazard rate order of the upper record values $X_{U_m}$ and $Y_{U_n}$ resulting from sequences of random lifetimes following the base distributions $F$ and $G$. The first result states that if $m \geq n$ and $X\succeq_{st} Y$, then $X\preceq_{c} Y$ implies $X_{U_m}\preceq_{c} Y_{U_n}$, which means that if $X$ ages faster than $Y$ in terms of the hazard rate function, then $X_{U_m}$ ages correspondingly faster than $Y_{U_n}$. In situations where $X_{U_m}$ and $Y_{U_n}$ have an IFR distribution, $X_{U_m}\preceq_{c} Y_{U_n}$ means that the hazard rate function of $X_{U_m}$ increases faster than the hazard rate function of $Y_{U_n}$. It is also concluded that if $Y_{U_n}\in IFR $ then $X_{U_m}\in IFR$. It has been shown that $X\succeq_{st} Y$ is not a necessary condition to obtain the preservation property of order $\preceq_{c}$ under upper record values. We have secondly given conditions under which the relative reversed hazard rate order between $X$ and $Y$ is translated into the relative reversed hazard rate order of lower record values $X_{L_m}$ and $Y_{L_n}$. We have shown that if $m \geq n$ and $X\preceq_{st} Y$, then $X\preceq_{b} Y$ results in $X_{L_m}\preceq_{b} Y_{L_n}$. If then $X_{L_m}$ and $Y_{L_n}$ have a DRHR distribution, then $X_{L_m}\preceq_{b} Y_{L_n}$ is equivalent to saying that the reversed hazard rate function of $X_{L_m}$ decreases faster than the reversed hazard rate function of $Y_{L_n}$. In this context, if $Y_{L_n}\in DRHR$ then $X_{L_m} \in DRHR$. We have also proved that $X\preceq_{st} Y$ is not a necessary condition to obtain the preservation property of order $\preceq_{b}$ under lower record values. \medskip

The results obtained in this study can be used to describe further findings on record values. The hazard rate function and reversed hazard rate function of a lifetime random variable $Z$ are proportional to the probabilities $P(X\leq t+\delta|X> t)$ and $P(X> t-\delta|X\leq t)$ respectively, where $\delta$ is a very small positive number. Usually it is important to predict future upper or lower records. Therefore, it is a natural question whether a record higher than $t$ will be reached directly after $t$. Let us assume $X$ and $Y$ have the cdfs $F$ and $G$ respectively. For example,
$$X_{U_m}\preceq_{c}Y_{U_n} ~\Rightarrow~\frac{P(X_{U_m}\leq t_1+\delta|X_{U_m}>t_1)}{P(Y_{U_n}\leq t_1+\delta|Y_{U_n}>t_1)} \leq \frac{P(X_{U_m}\leq t_2+\delta|X_{U_m}>t_2)}{P(Y_{U_n}\leq t_2+\delta|Y_{U_n}>t_2)},~~\textmd{for all}~t_1\leq t_2 \in \mathbb{R}^{+},$$
indicates that if the $m$th upper record from the sequence of i.i.d. lifetimes $X_1,X_2,\ldots$ from $F$ (i.e., $X_{U_m}$) has not retained until
time $t_2$ and also if the $n$th upper record from the sequence of i.i.d. lifetimes $Y_1,Y_2,\ldots$ from $G$ (i.e., $Y_{U_n}$) has not retained until
time $t_2$, then it is more likely in the time $t_2$ in comparison with the time $t_1$ that $X_{U_m}$ is equal to a value in a small neighborhood after $t_2$ than that $Y_{U_n}$ to be equal to such value in a small neighborhood after $t_2$.  Moreover, there is another question whether a lower record less than $t$ has retained right before $t$. We observe that
$$X_{L_m}\preceq_{b}Y_{L_n} ~\Rightarrow~\frac{P(X_{L_m}> t_1-\delta|X_{L_m}\leq t_1)}{P(Y_{L_n}> t_1-\delta|Y_{L_n}\leq t_1)} \geq \frac{P(X_{L_m}> t_2-\delta|X_{L_m}\leq t_2)}{P(Y_{L_n}> t_2-\delta|Y_{L_n}\leq t_2)},~~\textmd{for all}~t_1\leq t_2 \in \mathbb{R}^{+},$$
which shows that if the $m$th lower record from the sequence of i.i.d. lifetimes $X_1,X_2,\ldots$ from $F$ (i.e., $X_{L_m}$) has retained until
time $t_1$ and further if the $n$th lower record from the sequence of i.i.d. lifetimes $Y_1,Y_2,\ldots$ from $G$ (i.e., $Y_{L_n}$) has also retained before time $t_2$, then it is more likely in the time $t_1$ compared with the time $t_2$ that $X_{L_m}$ is equal to a value in a small neighborhood before $t_1$ than that $Y_{L_n}$ to be equal to such value in a small neighborhood before $t_1$.\medskip

In future research, one may consider simplifying the condition of Theorem \ref{thm2} with respect to $n,m,c_0$ and $c_1$ and also simplifying the condition of Theorem \ref{thm4} on the basis of $n,m,c_0^{\ast}$ and $c_1^{\ast}$. The existence of the supremum in
Theorem \ref{thm2} was guaranteed in Remark \ref{rem1} for the case that $X$ and $Y$ are equally distributed, where $c_0=c_1=1$
and also after the proof of Theorem \ref{thm4} it was guaranteed that the supremum exists in the cases where $X$ and $Y$ have a common distribution leading to $c^{\ast}_0=c^{\ast}_1=1$. However, as explained in Example \ref{exa7} and Example \ref{exa7777}, there are also more complicated situations in which the suprema exists. Further implications of the relative ordering of dataset values, e.g. finding bounds for the survival functions of the upper datasets and the distribution function of the lower datasets, can also be considered in the future.
\bigskip

\noindent \textbf{Acknowledgments}\medskip

\noindent This work was supported by Researchers Supporting Project number (RSP2024R392), King Saud University,
Riyadh, Saudi Arabia.\textbf{\medskip }

%
%
%

\noindent \textbf{Conflict of interest}\textbf{\medskip }

 On behalf of all authors, the corresponding author states that there is no conflict of interest.

\end{document}